\documentclass[a4paper, 12pt]{article}

\usepackage{amsmath}
\usepackage{amssymb}
\usepackage{amsthm}
\usepackage{mathtools}
\usepackage[notref,notcite]{showkeys}
\usepackage[utf8]{inputenc}
\usepackage[T1]{fontenc}
\usepackage{lmodern}
\usepackage[a4paper, margin=1in]{geometry}
\usepackage{graphicx} 
\usepackage{caption}  
\usepackage{subcaption} 

\usepackage{graphicx}
\usepackage{tikz}
\usepackage{pgfplots}
\pgfplotsset{compat=1.17}

\usepackage[colorlinks,linkcolor=red,citecolor=blue]{hyperref}
\usepackage{changes}

\usepackage{enumitem}
\usepackage{xcolor}
\usepackage{microtype}
\usepackage{physics}
\usepackage{array}
\usepackage{booktabs}

\newtheorem{theorem}{Theorem}[section]

\newtheorem{lemma}[theorem]{Lemma}

\newtheorem{prop}[theorem]{Proposition}

\newtheorem{remark}[theorem]{Remark}

\newcommand{\N}{\mathbb{N}}

\newcommand{\R}{\mathbb{R}}

\newcommand{\e}{\varepsilon}
\newcommand{\fa}{\forall}

\newcommand{\wtilde}{\widetilde}

\title{Strongly nonlinear age structured equation, time-elapsed model and large delays}

\author{Beno\^\i t Perthame\thanks{Sorbonne Universit{\'e}, CNRS, Universit\'{e} de Paris Cit{\'e}, Inria, Laboratoire Jacques-Louis Lions, LJLL, F-75005 Paris, France}
\and
Clément Rieutord\footnotemark[1] 
\and
Delphine Salort\thanks{Sorbonne Universit\'e, CNRS,  Computational Quantitative Synthetic Biology, UMR 7238. 4 Place Jussieu, 75005 Paris, France}
}

\date{\today}

\begin{document}
\maketitle


\begin{abstract} 
The time-elapsed model for neural networks is a nonlinear age structured equation where the renewal term describes the network activity and influences the discharge rate, possibly with a delay due to the length of connections.

We solve a long standing question, namely that an inhibitory network without delay will converge to a steady state and thus the network is desynchonised. Our approach is based on the observation that a non-expansion property holds true. However a non-degeneracy condition is needed and, besides the standard one, we introduce a new condition based on strict nonlinearity. 

When a delay is included, and following previous works for Fokker-Planck models, we prove that the network may generate periodic solutions. We introduce a new formalism to establish rigorously this property for large delays.

The fundamental contraction property also holds for some other age structured equations and systems.

\end{abstract} 
\vskip .4cm

\noindent{\makebox[1in]\hrulefill}\newline

\section{Introduction}

\paragraph{The nonlinear age structured equation.}
%
We consider the following nonlinear age structured equation.
\begin{equation}\label{eq:E}
    \left\{
\begin{aligned}
& \frac{\partial n}{\partial t}(t,x) + \frac{\partial n}{\partial x}(t,x)  + r(x,I(t))n(t,x) = 0, \quad t,\, x \geq 0,\\
& I(t) = \int_0^{\infty} r(x,I(t))n(t,x)dx, \\
& n(t,x=0) = I(t),\\
&n(t=0) = n_\text{ini}.
\end{aligned}
\right.
\end{equation}
This type of equation arises in many applications and here we are interested in neuroscience, \cite{gerstner2002spiking,SCHWALGER2019, GLP_book}. Then, $n(t,x)$ represents the probability density of finding a neuron at time $t$ such that the time elpased since the last discharge is~$x$ and~$r$ represents the rate of neuron discharge depending on the time elapsed since the last discharge and $I(t)$ the network activity, \cite{PPD,PPD2}. We assume that it satisfies, for some constant $r_M>0$,
\begin{equation} \label{as:r}
    \forall x\in(0,\infty),\, I\in [0,r_M], \quad 0 \leq r(x,I) \leq r_M, \qquad  \frac{\partial r}{\partial I}\in L^\infty((0,\infty)\times(0, r_M)).
\end{equation}

Since we often consider the inhibitory case, we introduce the condition
\begin{equation}\label{r_inhibit}
 \forall x\in(0,\infty),\; I\in [0,r_M], \qquad \frac{\partial r}{\partial I}(x,I) \leq 0.
\end{equation}

We also consider the effect of a transmission delay~$d>0$ in the rate $r(x,\cdot)$ which leads to the equation
\begin{equation} \label{eq:Ed}
      \left\{
\begin{aligned}
& \frac{\partial n}{\partial t}(t,x) + \frac{\partial n}{\partial x}(t,x)  + r\big(x,I(t-d))n(t,x\big) = 0, \quad t,x \geq 0,\\
& I(t) = \int_0^{\infty} r\big(x,I(t-d)\big)n(t,x)dx, \\
& n(t,x=0) = I(t),\\
&n(t=0) = n_\text{ini},\\
& I(t) = I_\text{ini} \leq r_M,  \quad \text{for} \quad t \in [-d,0).
\end{aligned}
\right.
\end{equation}

Both equations are conservative, and we assume that the initial data satisfies
\begin{equation} \label{as:initPr}
n_{\text{ini}} \in L^1(0,\infty), \qquad n_{\text{ini}} \geq 0, \qquad \int_{0}^{\infty} n_{\text{ini}}(x) \, dx = 1 .
\end{equation}
Then, for all $t> 0$, we have
\begin{equation} \label{as:proba}
n(t,x) \geq 0, \quad  I(t) \geq 0, \qquad \int_{0}^{\infty} n(t,x) \, dx =1 .
\end{equation}
It is useful for some statements to also assume that
\begin{equation} \label{asinit:strong}
K_{\text{ini}} :=\big \| \frac{\partial n_{\text{ini}}}{\partial x}(x)  + r\big(x,I(0)\big)n_{\text{ini}}(x)\big \|_1 <\infty.
\end{equation}
Here and throughout this paper, we have used the notation $\| u(x)\|_1 = \int_0^\infty |u(x)| dx$.

Our purpose is to show that, without delay, solutions of \eqref{eq:E} always converge as $t \to \infty$ to the unique steady state. Adding a large delay the solution is periodic.

\paragraph{Main results.} Detailed theorems are stated in each section and our main results are  summarized as follows.
\begin{theorem}[Long term convergence for inhibitory connections]
Assume that the rate function satisfies \eqref{as:r}, \eqref{r_inhibit} and $r(x,I)\geq r_0>0$, and the initial data satisfies \eqref{as:initPr}. Then the solution of  Eq.~\eqref{eq:E} converges exponentially fast to the unique steady state.
\end{theorem} 
The main novelty of this theorem is that no smallness on the rate function $r(x,I)$ is assumed. This is due to a global non-expansion property stated in Prop.~\ref{Prop:contraction}.

\begin{theorem}[Long time delays]
Assume that the rate function satisfies \eqref{as:r}, \eqref{r_inhibit}, $r(x,I)\geq r_0>0$ and the initial data satisfies \eqref{as:initPr}. Assume also that $\partial_I r(x,I)$  is large enough (see Section~\ref{sec:experiod}), then there is a $2d$ periodic function $\overline{ n}(t,x)$ oscillating between two states $I_\pm$, such that the solution of Eq.~\eqref{eq:Ed} with $I_\text{ini}=I_-$ or $I_+$ satisfies, for all $A>0$
\[
\frac 1 d \int_0^{Ad}  \| n(\frac t d,\cdot)-\overline{ n}(\frac t d, \cdot) \|_1 \to 0, \quad \text{as} \quad  t \to \infty.
\]
\end{theorem} 
The Cesaro mean convergence cannot be replaced by a pointwise convergence because of large oscillation around time discontinuities of $\overline{ n}(\frac t d, \cdot)$, see Fig.~\ref{fig:image2}.

\paragraph{Motivations, time-elapsed model of neuron assemblies.}
The analysis of nonlinear age structured equations is an old topic, often motivated by applications to biologucal systems. In that area, a cornerstone has been the book~\cite{MetzDiekmann_LN} and more recent books are \cite{IannelliBook, BP07}. A particular case is used for neuron assemblies by the time elapsed since the last discharge. This modeling has been proposed in ~\cite{PPCV} and the formalism was then studied in mathematical terms in \cite{PPD, PPD2} mostly using Lyapunov functionals (relative entropy) methods. The use of Doeblin method has been revisited in \cite{CH2019} and modern spectral methods have been used in~\cite{ Mischler_TENN_2018, Mischler_ws_2018}. In this literature, most of the results are based on linear stability analysis (see~\cite{CaceresCanizoT2025} for instance) and perturbation of the linear theory, in particular for what concerns long term behviour of the equations at hand. One of the major difficulties relies on the oscillatory behaviour for such networks.
\\

 For that reason, there are very few nonlinear age structured equations where global long term convergence or global Lyapunov functional are known. We are aware of simple cases where coeficients are such that the equation is reduced to a system of 2 or 3 ordinary differential equations, \cite{GMC79}. More elaborate is the Kermack–McKendrick model of epidemiology where a $log$ entropy has been established, see~\cite {MMW2010}. Another case is the equation
\[
\left\{
\begin{aligned}
& \frac{\partial n}{\partial t}(t,x) + \frac{\partial n}{\partial x}(t,x)  + r(x)n(t,x) = 0, \quad t,\, x \geq 0,\\
& n(t,x=0)  = F\left(  \int_0^{\infty} b(x) n(t,x)dx\right), 
\end{aligned}
\right.
\]
with a concave nonlinearity $F(\cdot)$, see~\cite {MichelNL07}. 
\\

Delays are known to generically produce oscillations, see \cite{murray1}. However, the process we use here is different and is based on a limit for long delays. For the integrate-and-fire model, it has been observed in~\cite{IRSS2022} that long delays may produce periodic solutions even for inhibitory connections. A clear explanation has been elaborated in~\cite{CaceresCanizo2024}, still for the integrate-and-fire model. We follow and adpat this point of view here. We note that there are direct connections between the time-elapsed and Integrate-and-fire models, see~\cite{DH1}.
\\

Motivations in terms of stochastic processes can be found in~\cite{CCDR}. they make the link between the time-elapsed model, which can be seen as a macroscopic model of neuron populations, and spike trains modelised by Poisson or Hawkes processes. Indeed, they show that the function $n(t,x)dx$ can be seen as the law of a random variable which represents the time elapsed since  its last discharge.

\paragraph{Organisation of the paper.}

We begin, in Section~\ref{sec:TEM}, with the inhitory case and no delay. We establish the non-expansion property (non-expansion) and use it to prove that the solutions will converge in long time to the unique steady state.

The delay is introduced in Section~\ref{sec:periodic}. First, we show that for a weak nonlinearity the solution will still converge to the unique steady state.  For long delays $d$, the solution will approach a limit described by iterates of nonlinear function. In particular for strong nonlinearities, these iterates are oscillating between two values. These phenomena are illustrated numerically.

Other examples with the same non-expansion property are mentioned in Section~\ref{sec:extensions}. In an appendix we recall well-posedness properties. 

\section{The time-elapsed model with inhibitory connections}
\label{sec:TEM}

A long standing question for Eq.~\eqref{eq:E}, is to know if inhibitory connections are enough to garantee the long term convergence to a steady state, a property which can be interpreted as the desynchonisation of the network. We answer positively to this question by establishing a non-expansion property.

\subsection{Non-expansion property}
\label{sec:contract}

We consider two solutions of Eq.~\eqref{eq:E} and for simplicity, we use th enotation $r_1 = r\big(x,I_1(t)\big)$ and $r_2 = r\big(x,I_2(t)\big)$.

\begin{prop} [Contraction property] \label{Prop:contraction}
Assume that $r$ satisfies \eqref{as:r}, \eqref{r_inhibit} and $n_\text{ini}$ verifies \eqref{as:initPr}. Let $n_1,\,n_2$ be two solutions of \eqref{eq:E} with initial data 
 $n_1^0$ and  $n_2^0$. Then we have, for all $t\geq 0$, 
\begin{equation}\label{contraction}
     \frac{d}{dt}\|n_1(t)-n_2(t)\|_1 =-G(t) \leq 0, 
\end{equation}
with
$$G(t) = \int_0^{\infty}\big[r_1|n_1-n_2|+ |r_1-r_2|n_2 \big] \big[ 1-\text{sg}({n_1}-{n_2})\text{sg}({I_1}-{I_2})\big]dx \geq 0 .
$$

Furthermore, assuming also \eqref{asinit:strong}, we have 
\begin{align} \label{timeDecay}
\frac{d}{dt}\big \| \frac{\partial n_i(t) }{\partial t}  \big\|_1 \leq 0 \qquad \text{and}\qquad \big  \|\frac{\partial n_i(t) }{\partial t} \big  \|_1 \leq K_\text{ini}.
\end{align}
\end{prop}

\begin{proof}
First, we show the inequality \eqref{contraction}. For $\varepsilon >0$, we define $\varphi_{\varepsilon}(z) = \sqrt{z^2+\varepsilon^2}$. Then, we compute
$$ 
\frac{\partial \varphi_{\varepsilon}(n_1-n_2)}{\partial t} +\frac{\partial \varphi_{\varepsilon}(n_1-n_2)}{\partial x} +(r_1n_1-r_2n_2)\varphi'_{\varepsilon}(n_1-n_2) =0.
$$
Next, integrating this equation on $(0,\infty)$, we get
$$ 
\frac{d}{d t}\int_0^{\infty} \varphi_{\varepsilon}(n_1-n_2)dx+\int_0^{\infty} [r_1(n_1-n_2)+n_2(r_1-r_2)] \varphi'_{\varepsilon}(n_1-n_2)dx = \varphi_{\varepsilon}(n_1-n_2)(t,0).
$$
Now, letting $\varepsilon \rightarrow 0$, we obtain
\begin{align} \label{eq:contract1}
\frac{d}{d t}\int_0^{\infty} |n_1-n_2|dx+\int_0^{\infty} [r_1|n_1-n_2| +n_2(r_1-r_2)\text{sg}(n_1-n_2)]dx = |I_1-I_2|(t).
\end{align}
Next, we use that
\[
|I_1-I_2| = (I_1-I_2)\text{sg}(I_1-I_2) = \int_0^{\infty}[(r_1-r_2)n_2+ r_1(n_1-n_2)] \text{sg}(I_1-I_2)dx,
\]
which is also 
\[
|I_1-I_2|=\int_0^{\infty} r_1(n_1-n_2) \text{sg}(I_1-I_2)dx- \int_0^{\infty}|r_1-r_2| n_2 dx  .
\]
Inserting this equality in \eqref{eq:contract1}, we obtain
\begin{align*}
\frac{d}{d t}\int_0^{\infty} |n_1-n_2|dx+\int_0^{\infty} \big[r_1|n_1-n_2| &- (r_1-r_2)n_2\;  \text{sg}(n_1-n_2)\big]dx= 
\\
&
 \int_0^{\infty} \big[-|r_1-r_2| n_2 +(n_1-n_2) \text{sg}(I_1-I_2)\big] dx.
\end{align*}
Rearranging the terms of this equality, we obtain \eqref{contraction}.
\\

Next, we turn to the proof of \eqref{timeDecay}. We differentiate Eq.~\eqref{eq:E} in time and to simplify we set $\partial_t n= \frac{\partial n}{\partial t}$. We obtain
    \[
    \partial_t (\partial_t n) +\partial_x(\partial_t n)+ \partial_t(rn) =0.
    \]
    and thus
    \[
    \partial_t |\partial_t n| +\partial_x|\partial_t n|+ \partial_t(rn)  \; \text{sg}(\partial_t n) =0.
    \]
    After integration, we find 
    \[
    \frac{d}{dt} \int_0^\infty|\partial_tn|dx-|\dot I|+\int_0^{\infty} \partial_t(rn)\text{sg}(\partial_tn)dx=0
    \]
because 
\[
|\dot I| =\text{sg}(\dot I) \dot I =\text{sg}(\dot I)\int_0^{\infty}\partial_t(rn)dx.
\]
Therefore,  we obtain
\[\begin{aligned}
        \frac{d}{dt} \int_0^\infty|\partial_tn|dx &= \int_0^{\infty} \big(r\partial_t n+ n\partial_t r \big)\big(\text{sg}(\dot I)-\text{sg}(\partial_tn)\big)dx
        \\
        &=  \int_0^{\infty} \big[
        \underbrace{- r |\partial_t n| +r \, \text{sg}(\dot I)\partial_t n}_{\leq 0}  \; 
        \underbrace{- n |\partial_t r| + n\, \text{sg}(\partial_t n)\partial_t r}_{\leq  0} \big] \leq 0.
\end{aligned} \]
Hence we get:
\[ \frac{d}{dt} \int_0^\infty|\partial_tn|dx \leq 0.\]
And then by integrating in $t$, we find 
\[\|\partial_tn(t)\|_1 \leq \|\partial_n(t=0)\|_1 = \|\partial_xn_\text{ini}+r(.,I(0))n_\text{ini}\|_1 = K_\text{ini}.
\]    
\end{proof}

\subsection{Bounds and regularity of the activity function}

The solutions of Eq.~\eqref{eq:E} satisfy some elementary a priori bounds which we recall here.

\begin{prop}[A priori bounds] \label{prop:reg} Assume that $r$ satisfies \eqref{as:r}, \eqref{r_inhibit} and $n_\text{ini}$ verifies \eqref{as:initPr}, \eqref{asinit:strong}. Let $n$ be a solution of~\eqref{eq:E} with activity function~$I$. Then, $I$ is Lipschitz continuous
$$
    \sup_{t \geq 0} I(t) \leq r_M\qquad \text{and}\qquad \sup_{t \geq 0}|I'(t)|\leq r_MK_\text{ini}.
$$
\end{prop}

\begin{proof}
For every \( t \geq 0 \), we have using \eqref{as:r} and \eqref{as:proba}
\[
    I(t) = \int_0^{\infty} r(x,I(t)) n(t, x) \, dx \leq r_M \int_0^{\infty} n(t, x) \, dx =
    r_M.
\]
Next, differentiating the equation defining $I(t)$, we obtain
\[
    I'(t) = \int_0^{\infty} \big[ I'(t) \, \partial_I r(x,I(t)) \, n(t, x) + r(x,I(t)) \, \partial_t n(t, x) \big] dx,
\]
\[
    I'(t) \big[ 1 - \int_0^{\infty} \partial_I r(x,I(t)) \, n(t, x) \, dx \big] = \int_0^{\infty} r(x,I(t)) \, \partial_t n(t, x) \, dx.
\]
Therefore, using the inhibitory assumption \eqref{r_inhibit},  we conclude that for all $t\geq 0$
\[
    |I'(t)| \leq |I'(t)| \left( 1 + \int_0^{\infty} |\partial_I r(x,I(t))| \, n(t, x) \, dx \right) \leq r_M \, \|\partial_t n(t)\|_1.
\]
    Then, according to \eqref{timeDecay}, we conclude that
    \[|I'(t)| \leq r_M K_\text{ini}\].
\end{proof}

\subsection{Stationary state}
\label{sec:stst}

Existence and uniqueness of a steady state can be reduced to simple fixed point argument. To do so, assuming \eqref{as:r}, we use the notation
\begin{equation} \label{def:R}
R(x,I) = \int_0^xr(y,I)dy \leq r_M x \qquad \qquad \text{(a decreasing function of $I$)}.
\end{equation}
We assume,
\begin{equation} \label{as:r2}
\int_0^{\infty}e^{-R(x,r_M)}dx <\infty, \qquad \text{and} \qquad   \sup_{I \in (0,r_M)} \int_0^{\infty} |\partial_I R(.,I)|\, e^{-R(.,I)} dx <\infty.
\end{equation}
We also introduce the definitions
\begin{equation}\label{def:agestst}
\begin{cases}
\displaystyle \Phi(I) := \bigg(\int_0^{\infty}e^{-R(x,I)}dx\bigg)^{-1} \qquad \Phi:[0,r_M]\to [0,r_M].
\\[5pt]
\displaystyle \overline{n}(x,I)= \Phi(I)e^{-R(x,I)}, \qquad \text{and thus} \quad \int_0^\infty \overline{n}(x,I) dx=1.
\end{cases}
\end{equation}
The function $\Phi$ is Lipschitz and non-increasing. 

We can now state the main result of this subsection.
\begin{prop}
Assume that $r$ satisfies \eqref{as:r}, \eqref{r_inhibit} and \eqref{as:r2}. Then, there is a unique mass~$1$ stationary state $(\overline{n}, \overline{I})$ for the inhibitory time-elapsed model~\eqref{eq:E} and it is given by 
$$
 \overline{n} (x) = \overline{I}e^{-R(x,\overline{I})}, \qquad  \text{with $\overline{I}$ a fixed point} \quad
\overline{I} =\Phi(\overline{I}).
$$    
\end{prop}

\begin{proof}
The stationary state $\overline{n}$ satisfies the differential equation
\begin{equation}
 \left\{
    \begin{aligned}
        &\frac{\partial}{\partial x} \overline{n}(x) + r(x,\overline{I}) \overline{n}(x) = 0, \quad x \geq 0,
    \\[5pt] 
        &\overline{n}(0) = \overline{I} = \int_0^\infty r(x,\overline{I}) \overline{n}(x)dx, \qquad \int_0^\infty \overline{n}(x) dx=1 .
    \end{aligned}
\right.
\end{equation}
Therefore, integrating the differential equation, the function $\overline{n}$ is given, for some $\overline{I}$ by 
$$
\overline{n}(x) = \overline{I} e^{-R(x,\overline{I})} 
$$
and thus the activity  $\overline{I}$ is determined by the condition
$$
\int_0^{\infty} \overline{n}(x)dx = 1 \Longleftrightarrow \overline{I} = \left(\int_0^{\infty} e^{-R(x,\overline{I})}dx\right)^{-1}= \Phi(I).
$$
The continuous function $I\mapsto \Phi (I)$ is non-increasing, positive for ${I}=0$ and, since   $ \Phi(I)\leq r_M$, there is a unique fixed point $I = \Phi (I)$ which uniquely determines the function~$\overline{n}=n(\overline{I})$.
\end{proof}

\subsection{Convergence towards the stationary state}

Depending on the long term convergence result, we use other assumptions on $r$. For exponential decay we use that for some $r_0 >0$ and all $x\geq 0$, $I\in [0,r_M]$, it holds
\begin{equation} \label{as:r3}
    r(x,I)\geq r_0 .
\end{equation} 
For convergence with no rate, we use  that there is an $x_* \in \text{Int \big(supp($r(x,\overline{I}$)\big)}$ such that the function
\begin{equation} \label{as:r4}
I \mapsto r(x_*,I)\:\: \text{is strictly decreasing for} \qquad I\in [0, r_M].
\end{equation}

We may now state our first result
\begin{theorem}[Convergence to the steady state]  \label{th:cv inib steady state}
Assume that $r$ satisfies \eqref{as:r}, \eqref{r_inhibit} and $n_\text{ini}$ verifies \eqref{as:initPr} and consider the solution $n(t,x)$ of \eqref{eq:E}. 
\\
(i) If the assumption \eqref{as:r3} also holds, then
\begin{equation} \label{cv1}
     \|n(t)-\overline{n}\|_1 \leq e^{-r_0 t}\|n_\text{ini}-\overline{n}\|_1.
\end{equation}
(ii) If the assumption \eqref{as:r4} also holds, then
\begin{equation} \label{cv2}
\underset{t \rightarrow \infty}{\lim} \|n(t)-\overline{n}\|_1 = 0.
\end{equation}
\end{theorem}

Let us point out that the exponential convergence with assumption \eqref{as:r3} is not surprising, see \cite{PPD, PPD2, BP07}. Using assumption \eqref{as:r4} is, up to our knowledge, completely new and relies strictly on the nonlinearity of the equation.
\\

\begin{proof} For simplicity, we denote by \( \overline{r} \) the function \( x \mapsto r(x,\overline{I}) \) and by \( r \) the function \( (t, x) \mapsto r(x,I(t)) \).
\\

We begin with proving \eqref{cv1}. We apply Prop.~\ref{Prop:contraction},  choosing \( n_1 = n \) and \( n_2 = \overline{n} \). For every $t \in [0,+\infty)$, we have, using~\eqref{contraction}
\begin{align*}
\frac{d}{dt}\|n(t)-\overline{n}\|_1 &= -\int_0^{\infty}\big[r|n-\overline{n}|+|r-\overline{r}|\overline{n}\big] \big[1-\text{sg}({n}-{\overline{n}})\text{sg}({I}-{\overline{I}})\big]dx
\\
&\leq -r_0 \int_0^{\infty}|n-\overline{n}|\big[1 - \text{sg}({n}-{\overline{n}})\text{sg}({I}-{\overline{I}})\big] dx
\\
& =-r_0 \int_0^{\infty}|n-\overline{n}| dx + r_0\;  \text{sg}({I}-{\overline{I}}) \int_0^{\infty} (n-\overline{n}) dx .
\end{align*}
Because of \eqref{as:initPr}, both $n$ and $\overline n $ have mass $1$, and thus
\begin{align*}
\frac{d}{dt}\|n(t)-\overline{n}\|_1 \leq -r_0 \int_0^{\infty}|n-\overline{n}| dx.
\end{align*}
Applying the Gronwall Lemma, we obtain the exponential decay as stated in~\eqref{cv1}.
\\

Next we prove \eqref{cv2}. Without restriction, we may assume that, for the initial data, the regularity bound~\eqref{asinit:strong} holds true after a regularisation denoted by $n_{init, \e}$. Indeed we may write
\[
\| n(t)-\overline{n}\|_1 \leq \| n(t)-n_\e(t) \|_1+ \| n_\e(t)-\overline{n}\|_1,
\]
therefore, if the long term result holds true for $n_\e(t)$, we get
\[
\limsup_{t\to \infty} \| n(t)-\overline{n}\|_1 \leq \| n(t)-n_\e(t) \|_1 \leq \| n_{init}-n_{init, \e} \|_1,
\]
which can be made as small as we want. 

\bigskip
\noindent \textbf{First step.}  We consider the function
\[
  H(t) = \int_0^{\infty}\overline{r}|n-\overline{n}|\big[1-\text{sg}(n-\overline{n})\text{sg}(I-\overline{I})\big]dx \geq 0.
\] 
Let us show that 
\begin{equation} \label{Hlimit}
    \lim_{t \rightarrow \infty} H(t) = 0.
\end{equation} 

According to Prop.~\ref{Prop:contraction}, the function \( t \mapsto \|n(t) - \overline{n}\|_1 \) is positive and decreasing, so \( \lim_{t \rightarrow \infty} \|n(t) - \overline{n}\|_1 = L \) exists. Integrating for \(t\in (0,\infty)\) Eq.~\eqref{contraction}, we obtain
\[
    \int_0^{\infty} |n_\text{ini} - \overline{n}| \, dx - L =  \int_0^{\infty} G(t) \, dt \geq \int_0^{\infty} H(t) \, dt.
\]
In particular, the function \( H \) satisfies
\[
    \forall t \geq 0, \quad H(t) \geq 0 \quad \text{and} \quad \int_0^{\infty} H(t) \, dt \leq \int_0^{\infty} |n_\text{ini} - \overline{n}|  < \infty.
\]
Thus, to prove \eqref{Hlimit}, it is enough to show that \( H \) is uniformly continuous on $[0,+\infty)$.

To do so, we rewrite the function $H$ so as to eliminate the term $\text{sg}(n-\overline{n})$
\begin{align*}
    H(t) &= \int_0^{\infty} \overline{r} |n - \overline{n}| \, dx - \text{sg}(I-\overline{I})\int_0^{\infty}(\overline{r}\,n-\overline{r} \, \overline{n}) dx
\\
   &= \int_0^{\infty} \overline{r} |n - \overline{n}| \, dx - \text{sg}(I-\overline{I})\int_0^{\infty}(r\,n -\overline{r} \, \overline{n} + \overline{r}\,n-r \, n) dx.
\end{align*}
Using  that $\text{sg}(I-\overline{I})(\overline{r}-r) = |r-\overline{r}|$, we get
\begin{equation} \label{eqHgood}
    H(t) = \int_0^{\infty} \overline{r} |n - \overline{n}| \, dx -|I-\overline{I}|-\int_0^{\infty}|r-\overline{r}|n \,dx.
\end{equation}
Differentiating this expression of  \( H(t)\) in time, we obtain
\begin{align*}
 H'(t) = \int_0^{\infty} \overline{r} \, \partial_t |n - \overline{n}| \, dx - \text{sg}(I - \overline{I}) \, \frac{ d(I - \overline{I})}{dt} - \int_0^{\infty} \left( \partial_t |r - \overline{r}| n + |r - \overline{r}| \, \partial_t n \right) \, dx
\end{align*}
and thus
\begin{align*}
 |H'(t)| & \leq  \int_0^{\infty} \overline{r} |\partial_t n| \, dx + |\frac{dI}{dt}| + \int_0^{\infty} \left( |\frac{dI}{dt}\partial_I r| \,n + |r - \overline{r}| \, |\partial_t n |\right) \, dx
\\
&\leq r_M \|\partial_t n(t)\|_1 + \big|\frac{dI}{dt}\big| + \int_0^{\infty} \big|\frac{dI}{dt}\big|  |\partial_I r| n \, dx + r_M \|\partial_t n(t)\|_1.
\\
&\leq 2 r_M \|\partial_t n(t)\|_1 + \big|\frac{dI}{dt}\big| \big(1+ \|\partial_I r\|_{\infty} \big).
\end{align*}
Consequently, using \eqref{timeDecay} and Prop.~\ref{prop:reg}, we have proved the Lipschitz bound 
\[
    \forall t \geq 0, \qquad |H'(t)| \leq 2 r_M K_{\text{ini}} +  r_M K_{\text{ini}} \left( 1 + \|\partial_I r\|_{\infty} \right)< \infty.
\]
And thus the limit \eqref{Hlimit} holds.
\\

\noindent \textbf{Second step.}
We now build a large time limit that we see through the sequence defined for $k \in \N$ by, 
\[
n_k(t,.) = n(t+k,.) \qquad \text{and} \qquad I_k(t) = I(k+t).
\]

We begin with the sequence $I_k$. By Prop.~\ref{prop:reg}, for all $T>0$ and $k\geq T$,  \(I_k(t)\) is Lipschitzian on $[-T, T]$. Using the Ascoli theorem and Cantor's diagonal argument, we can extract a subsequence (still denoted with index $k$) such that
\begin{equation} \label{limIk}
 \lim_{k\rightarrow\infty}\:\sup_{t\in [-T,T]}|I_k(t)-I_{\infty}(t)|=0 ,
\end{equation}
for some limit  \(I_{\infty}(.)\), a Lipschitzian function on \(\R\).
\

Now, let us show that the sequence $(n_k)_{k\geq T}$ is also compact in $\mathcal{C}^0([-T,T],L^1(0,\infty))$ for every $T>0$.
First we prove tightnes. Using the method of characteristics, we can write
\begin{align*}
n(t,x) = f(t,x)+\varepsilon(t,x),
\end{align*}
with
\[
\varepsilon(t,x) = n^0(x-t) \exp\left(-\int_{0}^t r(x-t+s,I(s))ds\right)\mathbf{1}_{x\geq t}
\]
and 
\[
f(t,x) = I(t-x)\exp\left(-\int_{0}^xr( s,I(t-x+s))ds\right)\mathbf{1}_{x\leq t}.
\]

Recalling the definition \eqref{def:R} and the property \eqref{timeDecay}, we can write, with $I_M= \sup I(\cdot) \leq r_M$, 
\[
\varepsilon (t,x) \leq \exp(-(R(x,I_M)-R(x-t,I_M)))n^0(x-t)\mathbf{1}_{x\geq t}
\]
and thus, for all $A>0$ and $t>0$, we have
\begin{align*}
      \|\varepsilon(t)\|_1 & \leq \int_0^{\infty} \exp\big(-(R(x+t,I_M) -R(x,I_M)  \big) n^0(x)dx
      \\
      & \leq \int_0^{A}\exp(-(R(x+t,I_M)-R(x,I_M)))n^0(x)dx +\int_A^{\infty}n^0(x)dx
     \\
    &\leq \exp(-(R(t,I_M) - R(A,I_M)))\int_0^A n^0(x)dx+\int_A^{\infty}n^0(x)dx.
\end{align*}

According to \eqref{as:r2}, $y\mapsto \exp(-R(y,r_M))$ is a non-increasing integrable function. Therefore, we have $\lim_{y\rightarrow \infty}\exp(-R(y,r_M)) = 0$.
Hence, 
$$\
\lim_{t\rightarrow \infty}\int_0^{A}\exp(-(R(x+t,I_M)-R(x,I_M)))n^0(x)dx = 0
$$
and so
\begin{align*}
     \fa A>0, \quad \limsup_{t\rightarrow \infty}\|\varepsilon(t)\|_1& \leq \int_A^{\infty}n^0(x)dx.
\end{align*}
Letting $A\rightarrow \infty$, we finally obtain
\[
\lim_{t\rightarrow \infty}\|\varepsilon(t)\|_1 = 0.
\]
Now, we write as before
\[
\varepsilon_k(t) = \varepsilon(t+k), \quad f_k(t) = f(t+k), \quad n_k(t) = f_k(t)+\varepsilon_k(t).
\]

On the one hand, thanks to \eqref{limIk}, we have, for all $A>0$ and the same subsequence as before,
\[
\lim_{k\rightarrow\infty}\:\sup_{t\in [-T,T], x\in [0,A]} \Big| f_k(t,x) - I_\infty(t-x)\exp\left(-\int_{0}^xr( s,I_\infty(t-x+s))ds\right) \Big| =0.
\]
We set
\begin{equation}\label{f_infini}
f_\infty(t,x) = I_\infty(t-x)\exp\left(-\int_{0}^xr( s,I_\infty(t-x+s))ds\right),
\end{equation}
Separating the integral $\int_0^\infty=\int_0^A +\int_A^\infty$ as before, we conclude that 
\begin{align*}
\limsup_{k\rightarrow \infty} \:\sup_{t\in [-T,T]} \|f_k (t)-f_\infty(t)  \|_1 \leq 2 \int_A^\infty f_k (t,x)dx.
\end{align*}
On the other hand, we can estimate, for all $A>0$,
\[
\int_A^\infty f_k (t,x)dx \leq I_M \int_A^\infty \exp \big(- R(x,r_M) \big) dx 
\]
which, as before according to \eqref{as:r2}, vanishes as $A \to \infty$.
\\
Altogether, we have proved that, for a subsequence 
\[
n_k\to n_\infty=f_\infty \quad \text{in} \quad C\big(-T, T; L^1(0,\infty)\big).
\]

\noindent \textbf{Third step.}
Next, we set $H_k(t)= H(t+k)\geq 0$. We use the form \eqref{eqHgood} for $H$ and we conclude from the above compactness arguments that, $H_k(t) \to H_\infty(t)$ uniformly on $[-T,T]$ for all  $T>0$ and
\[
H_\infty(t) := \int_0^\infty r(x,\overline{I}) |n_{\infty}(t)- \overline{n}|\,  \big[1-\text{sg}(n_\infty - \overline{n})\text{sg}(I_\infty-\overline{I})\big]dx .
\]
Finally because of \eqref{Hlimit}
\begin{equation}\label{Eq:H(infinity) = 0}
 H_k(t) =  H(t+k) \to H_\infty(t) =0 \qquad \text{as} \quad k \to \infty.
\end{equation}

\noindent \textbf{Fourth step.} We show that necessarily, $I_{\infty}(t) = \overline{I}$ and $n_{\infty}(t) = \overline{n}$ for every $t\geq 0$.\\

By contradiction, suppose that there exists $t_0\geq 0$ such that $I_{\infty}(t_0)>\overline{I}$ (a similar argument works for the other inequality). Then, let $J$ be the maximal interval such that $t_0 \in J$ and for all $ t \in J$, $I_{\infty}(t)>\overline{I}$.
\\
Assume without any loss of generality that $t_0 = 0$ and let $t_1 = \sup\{A>0 \quad / \quad I|_{[-A,A]}>\overline{I}\}$. If $t_1 = +\infty$, then, recall that we have the following formula for $n_{\infty}$
\[
\fa t\geq 0, x \in \R,\quad  n_{\infty}(t,x) = I_\infty(t-x)\exp\left(-\int_{0}^xr( s,I_\infty(t-x+s))ds\right),
\]
Then
\[\fa t,x \quad n_{\infty}(t,x)>\overline{I}\exp\left(-\int_{0}^xr( s,\overline{I})ds\right) = \overline{n}(x).\]
Then using the fact that $\int_0^{\infty}\overline{n}(x)dx = \int_0^{\infty}n(t,x)dx = 1$, we get a contradiction.
So $I(t_1) = \overline{I}$ or $I(-t_1) = \overline{I}$. We now assume that $I(t_1) = \overline{I}$.

Therefore,  according to the equality \eqref{Eq:H(infinity) = 0}, we have
\[
0 = \int_0^{\infty}r(x,\overline{I})|n_{\infty}(t_1,x)-\overline{n}(x)|dx.
\]
Using the continuity of $n_{\infty}$ and $\overline{n}$, we get
\begin{equation}\label{eq:t1}
\fa x \in \text{supp}(\overline{r}),\qquad  n_{\infty}(t_1,x) = \overline{n}(x).
\end{equation}
In addition, still due to \eqref{Eq:H(infinity) = 0} and sg$(I_\infty(t)-\overline{I})=1$, for $t \in]0,t_1[$, we have
\[
0 = \int_0^{\infty}r(x,\overline{I})\mathbf{1}_{\{n_{\infty}<\overline{n}\}}|n_{\infty}(t,x)-\overline{n}(x)|dx,
\]
which leads to 
\begin{equation}\label{ineq_n_inf-n_bar}
    \fa t \in ]0,t_1[,\quad \fa x \in \text{supp}(\overline{r}),\qquad  n_{\infty}(t,x)\geq \overline{n}(x).
\end{equation}

Now, let us take $x_* \in Int\{\text{supp}(\overline{r})\}$ and $\eta \in ]0,\frac{t_1}{2}[ $ such that $]x_*-2\eta,x_*+2\eta[\subset \text{supp}(\overline{r})$. We define the function $\phi$ such that 
\[
\fa s \in (-\eta, 3\eta),\qquad \phi(s) = n_{\infty}\left(t_1+s,x_*+\eta+s\right)-\overline{n}(x_*+\eta+s).
\]
Then, thanks to \eqref{eq:t1}, $\phi(0) = n_{\infty}(t_1,x_*+\eta)-\overline{n}(x_*+\eta) =\overline{n}(x_*+\eta)-\overline{n}(x_*+\eta) = 0$.\\

In addition, we have along the characteristics, for all $s \in , [-\eta,0],$
\begin{align*}
    \phi'(s) &=\overline{r}(x_*+\eta+s)\overline{n}(x_*+\eta+s) - 
    r(x_*+\eta+s,I_{\infty}(t_1+s))n_{\infty}(t_1+s,x_*+\eta+s)
\\
&\geq \overline{r}(x_*+\eta+s)\big(\overline{n}(x_*+\eta+s)-n_{\infty}(t_1+s,x_*+\eta+s)\big)
\\
&= -\overline{r}(x_*+\eta+s)\phi(s),
\end{align*}
using the inhibitory property \eqref{r_inhibit}. According to the Gronwall lemma, we deduce
\[
\fa s \in [-\eta,0], \qquad 0 =\phi(0)\exp\left(\int_s^0\overline{r}(x_*+\eta+u)du\right) \geq \phi(s).
\]
But according to the inequality \eqref{ineq_n_inf-n_bar}, we also have
\(0 \leq \phi(s)\) and thus we obtain
\[
\fa s \in [-\eta,0], \qquad \phi(s) = 0\qquad \text{and}\qquad \phi'(s) = 0,
\]
which implies, still for $s \in [-\eta,0]$
\[
 \big(\overline{r}(x_*+\eta+s)-r(x_*+\eta+s,I_{\infty}(t_1+s))\big)\overline{n}(x_*+\eta +s) = 0.
\]
Since $\overline{n}(x) = \overline{I}\exp(-R(x,\overline{I}))>0$, this also means
\[
\fa s \in [-\eta,0],\qquad \overline{r}(x_*+\eta+s)=r(x_*+\eta+s,I_{\infty}(t_1+s),).
\]
Choosing $s = -\eta$, we get
\[
r(x_*,\overline{I}) = \overline{r}(x_*) = r(x_*,I_{\infty}(t_1-\eta)).
\]
Using the strict monoticity of the function $I\mapsto r(I,x_*)$ in assumption \eqref{as:r4}, we get, for $\eta$ small enough
\[
I_{\infty}(t_1-\eta) = \overline{I}.
\]
Which contradict the hypothesis on $J = ]0,t_1[$.
\\

Therefore, the limit point \((n_{\infty}, I_{\infty}) = (\overline{n}, \overline{I})\) is unique. Thus, the solution \((n(t), I(t))\) converges to the stationary state \((\overline{n}, \overline{I})\), which proves \eqref{cv2}.
\end{proof}

\section{Solutions with a long delay}
\label{sec:periodic}

When  delay is included, that is we consider Eq.~\eqref{eq:Ed}, the theory changes drastically. With small nonlinearities, as expected, solutions behave as the linear equation. 

For long delays,
inspired by the paper~\cite{CaceresCanizo2024}, we build oscillatory solutions of the  time-elapsed equation  in the limit of large delays~$d$. Their idea is that the delayed activity $I(t-d)$, when $d$ is large, can be seen as a constant input which makes that solutions converge to a steady state. Then the problem is reduced to the study of the iterates of such a process. For this reason, we first state a convergence result for the linear equation before building the iterates and the periodic solutions. 

When this paper was under finalisation, we learned some related results have been proved in~\cite{CCTdelay}. In this paper, the assumptions on the rate function is more general to the expense of stronger assumptions on the initial data.
\\

The {\em linear age structured} equation is obtained when $J(t)\in C(0,\infty)$ is an input, 
\begin{equation}\label{eq:age}
    \left\{
\begin{aligned}
& \frac{\partial n}{\partial t}(t,x) + \frac{\partial n}{\partial x}(t,x)  + r(x, J(t))n(t,x) = 0, \quad t,\, x \geq 0,
\\
& I(t):=n(t,x=0) = \int_0^{\infty} r(x, J(t)) n(t,x)dx, 
\\
&n(t=0) = n_\text{ini}\geq 0, \qquad \int_0^{\infty}n_\text{ini}(x) dx =1.
\end{aligned}
\right.
\end{equation}

We use again the notation~\eqref{def:agestst} for the function $I\mapsto \Phi(I)$ and for $\overline{n}(x,\overline{J})=\Phi(\overline{J}) e^{-R(x\overline{J})}$, the mass one steady state obtained when $J(t)\equiv \overline{I}$ is constant in \eqref{eq:age}.

\subsection{Long term convergence for the linear equation}
\label{sec:linear}

Solutions of the linear Eq~\eqref{eq:age} converge towards the equilibrium state with an exponential rate when $r(x, I(t))$ also converges with an exponential rate. We state this in the following theorem.
\begin{theorem} [Exponential convergence for the linear eq.] \label{th:linearCV}
Assume that the rate function $r(x,I)$ satisfies \eqref{as:r},  \eqref{as:r3} and for some constants $\overline{J}, \, C_r>0$ and $\alpha\neq r_0$,
\begin{equation}\label{as:r5}
\fa t \geq 0, \qquad \|r(., J(t))- r(.,\overline{J})\|_\infty \leq C_r e^{-\alpha t}.
\end{equation}

Then, solutions of \eqref{eq:age} satisfy for all $t\geq 0$
\begin{equation} \label{ineqJ}
\|n(t) -\overline{n}(\cdot, \overline{J})\|_1 \leq 2\left(1+\frac{C_r}{|r_0-\alpha|}\right)e^{-\beta t}, \qquad \beta =\min( \alpha,  r_0),
\end{equation}
\begin{equation} \label{ineqJr}
| I(t) -\Phi(\overline{J})| \leq \left(2r_M\big(1+\frac{C_r}{|r_0-\alpha|}\big)+C_r\right)e^{-\beta t}.
\end{equation}
\end{theorem}
We notice that this result does not use the inhibitory property assumed in \eqref{r_inhibit}.

\begin{proof}[Proof of Theorem \ref{th:linearCV}]
We write $\overline{n}:=\overline{n}(x, \overline{J})$ and $r=r(x, J(t))$. As in Section~\ref{sec:contract}, $|n-\overline{n}|$ satisfies the equation
$$
\frac{\partial |n-\overline{n}|}{\partial t} + \frac{\partial |n-\overline{n}|}{\partial x}+ [rn-r(\overline{J})\overline{n}]\text{sg}(n-\overline{n}) = 0.
$$

Integrating in $x$ on $(0,\infty)$ we get,
\begin{align*}
\frac{d}{dt}\int_0^{\infty} |n-\overline{n}|&dx+\int_0^{\infty} [r(n-\overline{n})+ (r-r(\overline{J}))\overline{n}]\; \text{sg}(n-\overline{n})dx
\\
&
=\big|\int_0^{\infty} [r(n-\overline{n}))+ (r-r(\overline{J}))\overline{n}] dx \big| 
\\
&=\big| \int_0^{\infty} [(r-r_0)(n-\overline{n})+ (r-r(\overline{J}))\overline{n}] dx\big|
\end{align*}
because $n$ and $\overline{n}$ have mass one and $r_0$ is a constant. As a consequence, since $r-r_0\geq 0$, we obtain
\begin{align*}
\frac{d}{dt}\int_0^{\infty} |n-\overline{n}|dx+\int_0^{\infty} r |n-\overline{n}| \leq \int_0^{\infty} \big[(r-r_0)|n-\overline{n}| + 2 |r-r(\overline{J})|\; \overline{n}\big] dx.
\end{align*}
Simplifying the terms, we find
\begin{equation}\label{eq: estimation dn/dt avec r}
\frac{d}{dt}\int_0^{\infty} |n-\overline{n}|dx+\int_0^{\infty} r_0 |n-\overline{n}| dx \leq 2 \int_0^{\infty} |r-r(\overline{J})|\; \overline{n} dx \leq 2 \|r-r(\overline{J})\|_{\infty}.
\end{equation}

Since $r$ satisfies \eqref{as:r5}, we get
\[
\frac{d}{dt}\|n(t) - \overline{n}\|_1+ r_0\|n(t) - \overline{n}\|_1\leq  2C_re^{-\alpha t},
\]
and the Gronwall lemma gives, since $\alpha \neq r_0$, for all  $t\geq 0$, 
\begin{align}
\|n(t) - \overline{n}\|_1 &\leq e^{-r_0 t}\|n_\text{ini}-\overline{n}\|_1+2C_re^{-r_0 t}\int_0^{t}e^{(r_0-\alpha) \tau}d\tau \notag
\\
&\leq e^{-r_0 t}\|n_\text{ini}-\overline{n}\|_1+ 2C_r e^{-r_0 t}\frac{e^{(r_0-\alpha) t}-1}{r_0-\alpha}.
\label{linPrecise}
\end{align} 
The inequality \eqref{ineqJ} follows since \(\|n_\text{ini}-\overline{n}\|_1 \leq 2\).
\\

To prove \eqref{ineqJr}, we use that $I(t) = \int_0^{\infty}r(x,J(t))n(t,x)dx$ and $\Phi(\overline{J}) = \int_0^{\infty}r(x,\overline{J})n(x,\overline{J})dx$ and obtain, for all $t\geq 0$, 
\begin{align}
|I(t)-\Phi(\overline{J})|&\leq r_M\|n(t)-\overline{n}(.,\overline{J})\|_1+ \|r(.,J(t))-r(., \overline{J})\|_{\infty} \notag
\\
& \leq \big(2r_M(1+\frac{C_r}{|r_0-\alpha|})+C_r\big)e^{-\beta t}, \label{linPrecise2}
\end{align} 
which is the announced result.
\end{proof}

\subsection{Global convergence with a weak nonlinearity}

Using the method developed above, we establish the convergence to the steady state when the   nonlinearity is weak and with an arbitrary delay $d>0$. We are going to prove the
\begin{theorem} [Weak nonlinearity] \label{th:ETweak}
  Asssume that $n_\text{ini}$ satisfies \eqref{as:initPr} and that $r$ satisfies \eqref{as:r}, \eqref{as:r3}. Let $\gamma = \|\partial_I r\|_{\infty}$. Then, for $\beta $ such that  $\omega(\gamma):=\gamma \, \left(\frac{3 r_M}{r_0}+1 \right)<1$, we can find a constant $\lambda (\gamma)>0$ such that the solution $n$ of \eqref{eq:Ed} satisfies for all $t\geq0$
    \[
    |I(t)-\overline{I}|\leq \frac{2r_M}{\omega(\gamma)} e^{-\lambda t} \quad \text{and} \quad \|n(t)-\overline{n}\|_1\leq \frac{2}{\omega(\gamma)}  e^{-\lambda t}.
    \]
\end{theorem}

To do so we define
\begin{equation} \label{def:fg}
f(t) = \frac 1{r_M} |I(t)-\overline{I}|, \qquad  g(t) = \|n(t)-\overline{n}\|_1, 
\end{equation}
and show a preliminary result.
\begin{lemma}\label{lem: Estimation sur I et n delay wnl}
    With the assumption sof Theorem \ref{th:ETweak}, $f$ and $g$ satisfy the following system of inequalities
    \begin{equation}\label{eq:syst f et g}
    \forall t\geq s\geq 0, \qquad \left\{\begin{aligned}
        &f(t)\leq g(t)+\gamma f(t-d),\\
        &g(t) \leq g(s)e^{-r_0(t-s)} + 2\gamma r_M e^{-r_0t}\int_s^t e^{r_0\tau} f(\tau-d)d\tau.
    \end{aligned}
    \right.
    \end{equation}
\end{lemma}
\begin{proof}
    To prove the first inequality in \eqref{eq:syst f et g}, we use the following estimate
    \begin{align*}
    |I(t)-\overline{I}| &= \left|\int_0^{\infty} r(x,I(t-d))n(t,x)-r(x,\overline{I})\overline{n}(x)dx \right|\\ 
    &\leq \int_0^{\infty} n(t,x) |r(x,I(t-d))-r(x,\overline{I})| + r(x,\overline{I}) |n(t,x)-\overline{n}(x)|dx\\
    &\leq \|r(\cdot, I(t-d))-r(\cdot, \overline{I})\|_{\infty} + r_M\|n(t)-\overline{n}\|_1\\
    &\leq \gamma |I(t-d)-\overline{I}| + r_M\|n(t)-\overline{n}\|_1.
    \end{align*}
    To prove the second inequality, we use the estimate in \eqref{eq: estimation dn/dt avec r}
    \[
    \frac{d}{dt}\int_0^{\infty} |n-\overline{n}|dx + r_0 \int_0^{\infty} |n-\overline{n}| dx  \leq 2 \|r-r(\overline{I})\|_{\infty} \leq 2\gamma |I(t-d)-\overline{I}|.
    \]
    So, we obtain
    \[
    \forall t \geq 0, \qquad \frac{d}{dt} g(t)+r_0g(t)\leq 2 \gamma r_M f(t-d).
    \]
    Applying Gronwall’s lemma, we deduce
    \[
    \forall t\geq s \geq 0, \qquad g(t)\leq g(s)e^{-r_0(t-s)}+2\gamma r_M e^{-r_0t}\int_s^t e^{r_0\tau}f(\tau-d)d\tau,
    \]
    whcih concludes the proof of Lemma \ref{lem: Estimation sur I et n delay wnl}.
\end{proof}

\begin{proof}[Proof of Theorem \ref{th:ETweak}]  
We are going to estimate $f(t)$ and $g(t)$ by iterations on the time, departing with the obvious bound $f(t-d)\leq A$ and $g(t)\leq A$ for all  $t \geq 0$ and $A=2$.
    
    Using Lemma~\ref{lem: Estimation sur I et n delay wnl}, we get
    \[
    \forall t \geq s \geq 0, \quad g(t) \leq A \, \left(e^{-r_0(t-s)} + 2\gamma r_M e^{-r_0t} \int_s^t e^{r_0\tau}d\tau\right)\leq A \, \left(e^{-r_0(t-s)}+ \frac{2\gamma r_M}{r_0}\right).
    \]
    Now, setting $x_0 = \frac{1}{r_0} \log\left(\frac{r_0}{\gamma r_M}\right)$ and choosing $s = t-x_0$ for $t\geq x_0$, we obtain
    \[
    g(t) \leq A\, \frac{3\gamma r_M}{r_0}.
    \]
    Reinserting this bound into the inequality for $f(t)$, we get
    \[
    \forall t\geq x_0, \quad f(t) \leq A\,  \frac{3\gamma r_M}{r_0} +A\, \gamma = A \,  \gamma \, \left(\frac{3 r_M}{r_0}+1 \right).
    \]
    Defining $\omega(\gamma) =\gamma \, \left(\frac{3 r_M}{r_0}+1 \right)<1$, we obtain
    \[
    \forall t \geq x_0, \qquad f(t) \leq A \omega(\gamma) \quad \text{and} \quad g(t) \leq A \omega(\gamma).
    \]
    Applying the same process to $\tilde{g}(t) = g(t+x_0+d)$ and $\tilde{f}(t) = f(t+x_0+d)$, defining $t_0 = d+x_0$, and proceeding by induction, we obtain
    \[
    \forall n \in \mathbb{N}, \forall t\geq nt_0, \quad f(t)\leq A \omega(\gamma)^n \quad \text{and} \quad g(t) \leq A \omega(\gamma)^n.
    \]
    Finally, setting $\lambda=\frac{|\log \omega(\gamma)|}{t_0} > 0$, we conclude
    \[
    \forall t \geq 0, \quad f(t) \leq \frac{A}{\omega(\gamma)}e^{-\lambda t} \quad \text{and} \quad g(t) \leq \frac{A}{\omega(\gamma)}e^{-\lambda t},
    \]
    and Theorem \ref{th:ETweak} is proved.
\end{proof}

\subsection{Convergence of iterates when $d \rightarrow \infty$}
\label{section:dtoinfty}

We study now the behaviour of the solution of the time-elapsed delayed Eq.~\eqref{eq:Ed} when the delay is large and show its convergence. We follow the idea in \cite{CaceresCanizo2024}. Using Theorem~\ref{th:linearCV}, when one uses a constant initial input~$I_\text{ini}$ in Eq.~\eqref{eq:age}, for $d$ large the solutions $n(t,x)$, $0\leq t\leq d$,  for $t\gg 1$ is exponentially close to the steady state $\overline{n}(x, I_\text{ini})$. This function defines an input
$$
I_1(t)=n(t,0) \approx \Phi(I_\text{ini}) \qquad \text{for } t\gg 1.
$$
It will serve for $d\leq t \leq 2d$ as a delayed input for the linear Eq.~\eqref{eq:age}, which solution converges to $\overline{n}(x, I_1)$, which itself generates an input $n(t,0)\approx \Phi(I_1)$ servind for the next time intervalle and so on. For this process, the difficulty is to guarantee the exponential rates of convergence and to get rid of the transient. 

To state the precise result, for $d >0$ we denote by $n_d$ the solution of~\eqref{eq:Ed} and $I_d$ its activity function. The result is better stated in a renormalized time variable $\tau= \frac t d$ and we set
\begin{align}
\wtilde {n}_d(\tau) = n_d(\tau {d}), \qquad  \wtilde{I}_d(\tau) = I_d(\tau{d}).
\end{align}

\begin{theorem}[Convergence of iterates] \label{cv phi iter}
We assume \eqref{as:r},  \eqref{as:r3} and \eqref{as:initPr}.  Let $n_d(t,x)$ be the solution of~\eqref{eq:Ed}, $I_d(t)=n_d(t,0)$ the associated activity function. We set $\Phi^0(I)\equiv I$ and define
\[
\wtilde {n}_{\infty}(\tau, x,I_\text{ini}) = \sum_{k=1}^{\infty} \overline{n}(x,\Phi^{k-1}(I_\text{ini})) \mathbf{1}_{k-1\leq \tau < k} \quad \text{and} \quad
\wtilde {I}_{\infty}(\tau, I_\text{ini}) = \sum_{k=0}^{\infty} \Phi^k(I_\text{ini}) \mathbf{1}_{k-1\leq \tau < k} .
\]
Then, for every $N \in \mathbb{N}$, we have

\begin{equation}\label{eq:convergence delai grand}
    \lim_{d\rightarrow \infty}\int_0^N \big\| \wtilde {n}_d(\tau)-\wtilde {n}_{\infty}(\tau, I_\text{ini})\big \|_1 d\tau=0 \quad \text{and} \quad \lim_{d\rightarrow \infty}\int_0^N\big|\wtilde{I}_d(\tau)-\wtilde {I}_{\infty}(\tau, I_\text{ini})\big|d\tau = 0.
\end{equation}
\end{theorem}

\begin{remark} \label{rk:iter} Assuming that $\gamma = \|\partial_I r\|_{\infty}<1$, the Picard iterates $Phi^{k}(I_\text{ini})$ converge to the uniquee fixed point of $\Phi$. But this is not enough to guarantee the long term convergence of $n(t,x)$. The weak interconnection condition in Theorem~\ref{th:ETweak} is stronger. See the numerical example on Fig.~\ref{fig:image3}. 
\end{remark}

\begin{proof}[Proof of Theorem \ref{cv phi iter}]
The proof uses two steps. The first is in the variable $t$, the second gives consequences in the variable $\tau$.
\\
{\em Step 1. Estimating the convergence rates on intervals [kd, (k+1)d)].}\\
For simplicity we fix a decay rate $\alpha$, $0 < \alpha < r_0$. We define, for every $k \geq 0$,
\[
C_k = \sup_{d > 0} \sup_{0 \leq t \leq d} 
\big( |I_d(t + (k-1)d) - \Phi^k(I_\text{ini})| e^{\alpha t} \big),
\]
and
\[
D_k = \sup_{d > 0} \sup_{0 \leq t \leq d} 
\big( \|n_d(t + kd) - \overline{n}(\cdot, \Phi^k(I_\text{ini}))\|_1 e^{\alpha t} \big).
\]
We have \( C_0 =0 \) because, by hypothesis on the initial condition,
\[
\forall d > 0, \, \forall t \in [0, d), \qquad I_d(t - d) = I_\text{ini}= \Phi^0(I_\text{ini}).
\]
Next we argue by induction. Assume that \( C_k < \infty \) for some \( k \geq 0 \). Then, we have,
for all $d>0$, $0\leq t \leq d$,
\begin{align*}
\|r(I_d(t + (k-1)d)) - r(\Phi^k(I_\text{ini}))\|_\infty 
&\leq \|\partial_I r\|_\infty |I_d(t + (k-1)d) - \Phi^k(I_\text{ini})|
\\[5pt]
&\leq \|\partial_I r\|_\infty C_k e^{-\alpha t}.
\end{align*}
Using Theorem \ref{th:linearCV}, with $J(t) = I_d(t+(k-1)d)$ and $\overline{J}=\Phi^k(I_\text{ini})$, we deduce that for all $d>0$, $0\leq t \leq d$,
\begin{align*}
\|n_d(t + kd) - \overline{n}(\cdot, \Phi^k(I_\text{ini}))\|_1 &\leq \left(2+2\frac{\|\partial_Ir\|_{\infty}C_k}{|r_0-\alpha|} \right)e^{-\alpha t},
\\
|I_d(t + kd) - \Phi^{k+1}(I_\text{ini})| & \leq \left(r_M \big(2+2\frac{\|\partial_Ir\|_{\infty}C_k}{|r_0-\alpha|}\big)+ \|\partial_I r\|_\infty C_k \right)e^{-\alpha t}.
\end{align*}
Hence, \( C_{k+1} \leq 2r_M+\|\partial_I r\|_\infty \big(\frac{2r_M}{|r_0-\alpha|} + 1\big)C_k < \infty \) and \( D_{k+1} \leq  2+2\frac{\|\partial_I r\|_\infty C_k}{|r_0-\alpha|}< \infty \). By induction, we conclude
\[
\forall k \geq 0, \qquad C_k < \infty \quad \text{and} \quad D_k < \infty.
\]
{\em Step 2. Estimating the norms in variable $\tau$.}
Since we have exponential convergence on every interval of the form \([kd, (k+1)d]\), we can estimate the normalized functions~\(\tilde{n}_d\) and~\(\tilde{I}_d\).To prove \eqref{eq:convergence delai grand}, we notice that for every~\(d > 0\), 
\begin{align*}
\int_0^N |\wtilde{I}_d(\tau) - \wtilde {I}_\infty(\tau, I_\text{ini})| \, d\tau 
&= \sum_{k=0}^{N-1} \int_k^{k+1} |\wtilde{I}_d(\tau) - \Phi^{k+1}(I_\text{ini})| \, d\tau,
\\
&= \frac{1}{d} \sum_{k=0}^{N-1} \int_0^d |I_d(t + kd) - \Phi^{k+1}(I_\text{ini})| \, dt,
\\
&\leq \frac{1}{d} \sum_{k=0}^{N-1} \int_0^d C_{k+1} e^{-\alpha t} \, dt 
\leq \frac{1}{d \alpha} \sum_{k=0}^{N-1} C_{k+1}.
\end{align*}
From the bounds in Step 1, we obtain
\[
\lim_{d \to \infty} \int_0^N |\wtilde{I}_d(\tau) - \wtilde {I}_\infty(\tau, I_\text{ini})| \, d\tau = 0.
\]
Using the same reasoning, we find
\[
\int_0^N \|\wtilde{n}_d(\tau) - \wtilde {n}_\infty(\tau, I_\text{ini})\|_1 \, d\tau 
\leq \frac{1}{d \alpha} \sum_{k=0}^{N-1} D_k {\longrightarrow} 0 \qquad \text{as } d \to \infty.
\]
\end{proof}

\subsection{Example of periodic solutions}
\label{sec:experiod}

According to Theorem \ref{cv phi iter}, the long time limit function \( n_{\infty} \) depends only on the initial data \( I_\text{ini} \) and on the behaviour of the iterates of function \( \Phi \) defined by \eqref{def:agestst}. Here we assume again the inhibitory property \eqref{r_inhibit} so that there is always a unique fixed point $\overline{I}=\Phi(\overline{I})$. 
Then, a periodic solution $\wtilde{n}_\infty$ is obtained if one can find  $I_\pm \neq \overline{I}$ such that $I_+=\Phi(I_-)$, $I_-=\Phi(I_+)$ so that the values of $\wtilde{n}_\infty$ oscillate between $I_+$ and $I_-$. This means that \( \Phi \circ \Phi \) should admit two fixed points away from $\overline{I}$. Such a property is equivalent to the property $\Phi'(\overline{I})<-1$ (see Fig. \ref{fig:image1}).

Since \( \Phi \) is entirely determined by the rate function \( r(I,x) \), to obtain such a periodic state,  it is enough to find a rate function such that \( \Phi \circ \Phi \) admits a fixed point that is not~$\overline{I}$. This can be achieved with the standard rate function introduced in \cite{PPD,PPD2}, with $r_0\geq 0$, 
\begin{equation} \label{def:rsimple}
r(x, I) = r_0 + \mathbf{1}_{x \geq \sigma(I)} , \qquad \Phi(I)=\left(\frac{1-e^{-r_0\sigma(I)}}{r_0}+\frac{e^{-r_0\sigma(I)}}{1+r_0}\right)^{-1} \in [r_0,r_0+1].
\end{equation}
Notice that, knowing the function $\Phi:[r_0, r_0+1]\to [r_0, r_0+1]$, one immediately recover~ $\sigma$, and thus the rate function, using the above relation. Therefore we give condtions 

\begin{prop}[Periodic states]\label{prop:Periodic States}
   Consider the specific case when $r$ and $\Phi$ are given by \eqref{def:rsimple} with $\sigma(\cdot)>0$ a differentiable increasing function. Then, \(\Phi(I)\) is a decreasing and differentiable function and
\\[5pt]
1. \(\Phi\) has a unique fixed point \(\overline{I}\).
\\[5pt]
2. If \(\Phi'(\overline{I}) < -1\), then there exist \(I_- < \overline{I} < I_+\) such that 
    \[
    \Phi(I_-) = I_+ \quad \text{and} \quad \Phi(I_+) = I_-,
    \]
and $\wtilde {n}_{\infty}(\tau, x,I_\text{ini}) $ is periodic of period $2$.    
\end{prop}

As a consequence of this property and Theorem~\ref{eq:convergence delai grand}, the solution $n(t,x)$ in the normalized time scale $\tau$ converges to a periodic solution as $d\to \infty$.  See Fig.\ref{fig:image2}.
\begin{proof}
   The first property follows from Section~\ref{sec:stst} because $\sigma(\cdot)$ increasing implies that $r(x, I)$ is  decreasing (inhibitory type).
   \\
    
    For the second property, we assume that $\Phi'(\overline{I}) <-1$.
    We define $\Psi = \Phi \circ \Phi$. Then, $\Psi$ is strictly increasing, and differentiable on $\mathbb{R}_+$.
    In addition, we have
    \[
         \Psi'(\overline{I}) = \Phi'(\overline{I})\Phi'(\Phi(\overline{I})) = \Phi'(\overline{I})^2 >1, \qquad \underset{I\rightarrow \infty}{\lim}\Psi(I) = \Phi\left(\underset{I\rightarrow \infty}{\lim}\Phi(I)\right) >0.
    \]
    Now, define $H(I) = \Psi(I)-I$. Then, we have
\[
H(\overline{I}) = 0 \quad \text{and} \quad H'(\overline{I}) = \Psi'(\overline{I})-1>0,
\]
therefore $H(I)<0$ for $I\lesssim \overline{I}$. Since $H(0) = \Psi(0) = \Phi(\Phi(0))>0$, by the intermediary value theorem, $H$ has to vansih between $0$ and $\overline{I}$ at a value denoted by $\overline{I}_-$. Obviously $\overline{I}_+=\Phi(\overline{I}_-)>\Phi(\overline{I})=\Phi(\overline{I}$ is another fixed point of $\Phi \circ \Phi$.


    
\end{proof}

\subsection{Numerical illustrations}

We have run numerical simulations in order to illustrate our main results on the effect of the delay. 
\\

For a parameters $\gamma$ to be chosen later and $r_0=0.5$, we use $\overline{I} = \frac{1+r_0}{2}>0$ and the decreasing function $\Phi(I) = \frac{|1-r_0|}{2}\tanh(\frac{2\gamma}{|1-r_0|}(\overline{I}-I)) +\overline{I}$, with $\gamma > 1$. Then, $\Phi$ admits an unique fixed point $\overline{I}$ such that $|\Phi'(\overline{I})| =\gamma $. 
\\

In a first case, we use $|\Phi'(\overline{I})| =\gamma>1$, according to Prop.~\ref{prop:Periodic States}, $\Phi\circ \Phi$ admits two fixed points $I_{\pm}\neq \overline{I}$ such that $I_+ = \Phi(I_-)$ and periodic solutions exist oscillating between $I_{\pm}$ (see Fig.~\ref{fig:image1}). We have computed the solution $\tilde{I}_d$ with $\gamma = 1.5$ and a large delay $d \approx 15\sigma(J)$. According to Theorem~\ref{cv phi iter},  solution is nearly periodic. See Fig.~\ref{fig:image2}.
\\

\begin{figure}[h]
    \centering
    \includegraphics[width=0.4\textwidth]{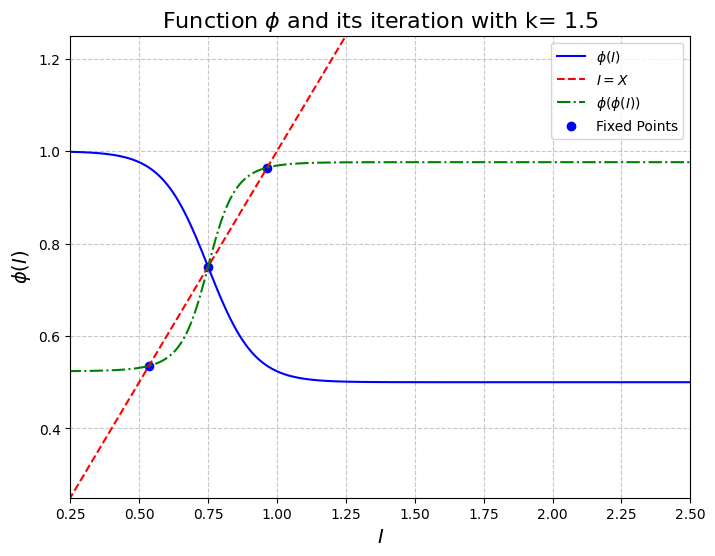}
    \\[-7pt]
    \caption{Graph of $\Phi$ and $\Phi\circ\Phi$. Here $\Phi'(\overline{I}) <-1$ and a periodic solution arises oscillating between the two extreme fixed points of $\Phi\circ\Phi$.}
    \label{fig:image1}
\end{figure}

\begin{figure}[h]
    \centering
    \includegraphics[width=0.4\textwidth]{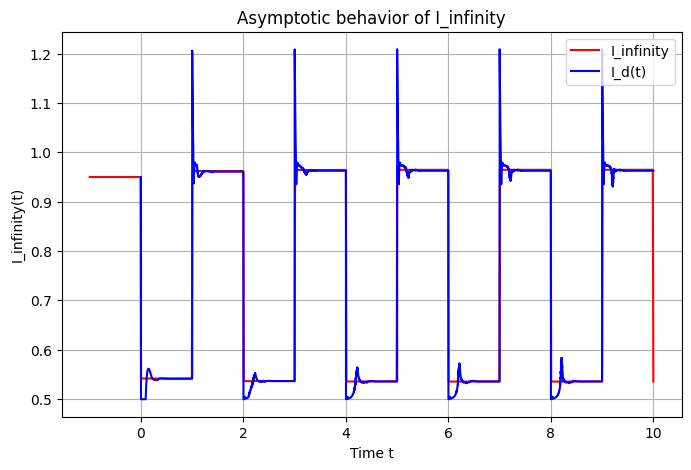}
    \\[-7pt]
    \caption{The solution $\tilde{I}_d(\tau)$ in blue and the function $\tilde{I}_\infty(\tau)$ in red. Here $d$ is large and the solution is periodic of period $2$ in the variable $\tau = \frac t d$.}
    \label{fig:image2}
\end{figure}

In a second case, we illustrate Remark~\ref{rk:iter} and Theorem~ \ref{th:ETweak}. Since $r_M =r_0+1$, we compute $\omega(\gamma) =\gamma (4+r_0^{-1}) = 6\gamma$ and we choose the parameters $\gamma \in (\frac{1}{6}, 1)$ and $d \approx 7.2\sigma(J)$. Then, the sequence $(\Phi^k(I_0))_{k\geq 0}$ converges to $\overline{I}$ while  the solution $n(t)$ may not. This phenomenon is observed numerically on Fig.~\ref{fig:image2}.
    
    \begin{figure}[h]
    \centering
    \includegraphics[width=0.4\textwidth]{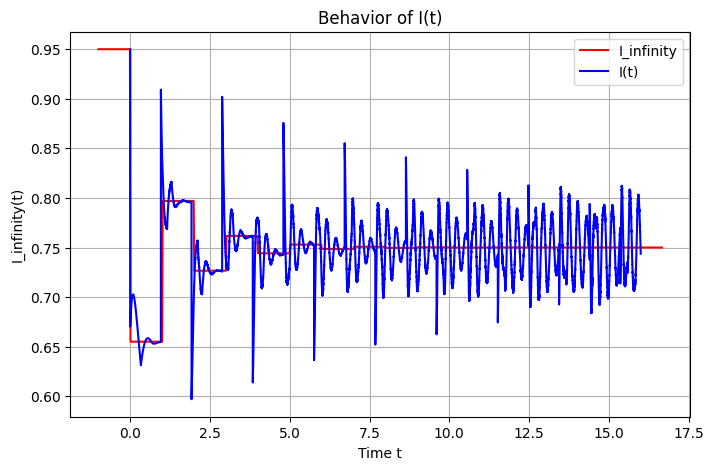}
    \\[-7pt]
    \caption{The solution $I(\tau)$ in blue and the function $\tilde{I}_\infty(\tau)$ in red. Here $d$ is moderate and a chaotic behaviour of $I(t)$ is observed while the iterates converge to $\overline{I}=0.75$.}
    \label{fig:image3}
    \end{figure}

\section{Other equations with a non-expansion principle}
\label{sec:extensions}

The non-expansion principle, and thus its consequences, can be extended to other age structured equations that we mention now. An example is with distributed birth, another is the case of systems. Here we only state the non-expansion principle and do not write the proof which is similar to that in Section~\ref{sec:contract}.

\subsection{Distributed birth}

The simplest extension of the age structured equation is with distributed birth. With the same notations of the introduction, it reads

\begin{equation}\label{eq:distr}
    \left\{
\begin{aligned}
& \frac{\partial n}{\partial t}(t,x) + \frac{\partial n}{\partial x}(t,x)  + r(x, I(t))n(t,x) = B(x) I(t), \quad t,\, x \geq 0,\\
& n(t,x=0) = 0, \qquad I(t) = \int_0^{\infty} r(x,I(t)) n(t,x)dx, \\
&n(t=0) = n_\text{ini}\geq 0, \qquad \int_0^{\infty}n_\text{ini}(x) dx =1.
\end{aligned}
\right.
\end{equation}
Here the distributed birth $B(x)$ is a probability measure and we still assume that the rate $r(x,I)$ satisfies the assumptions \eqref{as:r}, \eqref{r_inhibit} and thus is inhibitory. Notice that the time-elapsed equation is a special case when $B(x)=\delta _0(x)$, a Dirac mass.

As before, for two solutions $n_1$, $n_2$, we still have the non-expansion property
\begin{equation} \label{eq:distrcontract}
\begin{aligned}
\frac{d}{dt}  \int_0^{\infty}& |n_1(t,x)-n_2(t,x)| dx =
\\
&- \int_0^{\infty}\big[r_1|n_1-n_2|+ |r_1-r_2|n_2 \big] \big[ 1-\text{sg}({n_1}-{n_2})\text{sg}({I_1}-{I_2})\big]dx \leq 0. 
\end{aligned}
\end{equation}

Assuming that the initial data satisfies the stronger hypothesis~\eqref{asinit:strong}, we also have 
\begin{equation} \label{eq:distrBV}
\begin{aligned}
        \frac{d}{dt} \int_0^\infty |\partial_t n |dx = - \int_0^{\infty} [ r |\partial_t n| + n |\partial_t r | ][1- \text{sg}(\dot I) \text{sg}(\partial_t n) ] \leq 0.
\end{aligned}
\end{equation}

A consequence of this inequality is that, using Eq.~\eqref{eq:distr}, both $ |\partial_t n |$ and $ |\partial_x n | $ are bounded measures and thus $n(t,x)$ is a function and not only a measure. Because of the non-expansion property, without  hypothesis~\eqref{asinit:strong}, it remains that for $n_\text{ini} \in L^1(0,\infty)$, the solution $n(t,x)$ is  a function and not only a measure. Notice that the linear equation is also non-expensive for some Monge-Kantorowich distances~\cite{FoPeSIAM}.

\subsection{Example of a system}

The non-expansion principle also holds for some systems. The simplest example is the system of two equations ($i=1,\, 2$) with distributed birth
\begin{equation}\label{eq:sys}
    \left\{
\begin{aligned}
& \frac{\partial n_i}{\partial t}(t,x) + \frac{\partial n_i}{\partial x}(t,x)  + r_i(x, I_i(t))n(t,x) = B_i(x) I_{j}(t), \quad t,\, x \geq 0,\\
& n_i(t,x=0) = 0, \qquad I_i(t) = \int_0^{\infty} r_i(x,I_i(t)) n_i(t,x)dx, \\
&n_i(t=0) = n_{ini,i}\geq 0, \qquad \int_0^{\infty}n_{ini,i}(x) dx =1, 
\end{aligned}
\right.
\end{equation}
where $j=2$ when $i=1$ and $j=1$ when $i=2$.

In this case, the non-expansion principle is written on the sum of the $L^1$ norms of the difference between two couples $(n_1,n_2)$ and $(\wtilde n_1, \wtilde n_2)$.
\begin{equation} \label{syscontract}
\begin{aligned}
\frac{d}{dt}  \int_0^{\infty}&[| n_1(t,x)- \wtilde n_1(t,x)|+ | n_2(t,x)- \wtilde n_2(t,x)|]  dx =
\\
&- \int_0^{\infty}\big[r_1|n_1-\wtilde n_1|+ |r_1- \wtilde r_1| \wtilde n_1 \big] \big[ 1-\text{sg}({n_1}-{\wtilde n_1})\text{sg}({I_1}-{\wtilde I_1})\big]dx
\\
&- \int_0^{\infty}\big[r_2|n_2-\wtilde n_2|+ |r_2- \wtilde r_2| \wtilde n_2 \big] \big[ 1-\text{sg}({n_2}-{\wtilde n_2})\text{sg}({I_2}-{\wtilde I_2})\big]dx \leq 0. 
\end{aligned}
\end{equation}

\section{Conclusion and discussion}

Our analysis of the time-elapsed model covers two aspects. 

On the one hand, for inhibitory connections without delay, the non-expansion property allows us to show several long term convergence to the unique steady state. This holds for strong nonlinearities and a new non-degeneracy condition is introduced. Our analysis is extended to distributed birth and systems. However it leaves open several other classical equations as the multi-time age structured equation, see~\cite{TPS2022}, or the age structured model with leaky memory~\cite{FonteS2022}.
\\

On the other hand, when delay is included, weak nonlinearities will lead to long term relaxation towards the unique steady state, what ever is the delay. For long delays and strong enough nonlinearities, periodic solutions may emerge following a generic process of iterative relaxation for the linear equation. Then the period turns out to be $2d$. Observations are that spontaneous activity occurs with a period much larger than the single neuron time scale, thus we leave open to understand different processes.

\section{Appendix}
\begin{theorem}
    Assume $r$ satisfies \eqref{as:r}, \eqref{r_inhibit} and $n_\text{ini}$ satisfies \eqref{as:initPr}. Then, there exists a unique solution $n\in \mathcal C((0,\infty; L^1(0,\infty))$ of \eqref{eq:E}.
\end{theorem}

\begin{proof}
    For greater clarity, we define $\mathcal{M}^+_1(0,\infty) = \{f\in L^1(0,\infty), \:\: f\geq 0\:\: a.e. \:\text{and}\:\|f\|_1=1\}$. Let $T>0$ and $n_\text{ini} \in \mathcal{M}^+_1(0,\infty)$. We define 
    $E = \mathcal{C}^0\big([0,T], L^1(0,\infty)\big)$  and  $X = \mathcal{C}^0\big([0,T], \mathcal{M}^+_1(0,\infty)\big)$. We also write $\forall \mu \in E, \:\|\mu\|_E=\sup_{t\in [0,T]}\|\mu(t)\|_1$. Notice that $(E,\|.\|_E)$ and $(X,\|.\|_E)$ are Banach spaces and  for every $n \in X, \|n\|_E=1$. \\

    We also extend the function $r$ initially defined on $[0,\infty)\times[0,r_M]$ to $[0,\infty)\times\R$,  setting 
    \[
    r(I) = 
    \begin{cases} 
        r(r_M) & \text{if } I \geq r_M, \\
        r(0) & \text{if } I \leq 0.
    \end{cases}
\]
Such a definition ensures that $r$ still satisfies \eqref{as:r} and \eqref{r_inhibit}.
Now we proceed in 2 steps:
\begin{enumerate}
    \item Firstly, for any measure $n\in X$, we find an activity function $I(t) \in \mathcal{C}^0([0,T],\mathbb{R})$ and an other measure function $\mu(n)\in X$ such that :
    $$\left\{\begin{aligned}
    & \frac{\partial \mu(n)}{\partial t}(t,x) + \frac{\partial \mu(n)}{\partial x}(t,x)  + r(x,I(t))\mu(n)(t,x) = 0, \quad t,\, x \geq 0,\\
    & I(t) = \int_0^{\infty} r(x,I(t))n(t,x)dx, \\
    & \mu(n) (t,x=0) = \int_0^\infty r(I(t),x)\mu(n)(t,x)dx\\
    &\mu(n)(t=0) = n_\text{ini}.
    \end{aligned}
    \right.$$
    \item Secondly, we show that the function $n \mapsto \mu(n)$ is a contraction in $X$ for $T$ small enough.
\end{enumerate}

Let \[
\begin{aligned}
\Phi: X \times \mathcal{C}^0([0,T],\mathbb{R}) &\longrightarrow \mathcal{C}^0([0,T],\mathbb{R}) \\
(n, I) &\mapsto \left[ t \mapsto I(t) - \int_0^\infty r(I(t),x) n(t,x) \,dx \right]
\end{aligned}.
\]
For every $h\in \mathcal{C}^0([0,T],\mathbb{R})$, $\partial_I\Phi(n,I)\cdot h = \big(1-\langle \partial_Ir(I),n\rangle\big).h$ .
Notice that for every $I\in \mathcal{C}^0([0,T],\mathbb{R}^+)$ and $n \in X$, $1-\langle \partial_Ir(I),n\rangle =1+\langle |\partial_Ir(I)|,n\rangle\geq 1$.
So 
$$\forall t \in [0,T], \qquad |\partial_I\Phi(n,I)\cdot h(t)| = \big(1+\langle |\partial_Ir(I)|,n\rangle\big)|h(t)|\geq |h(t)|.$$
So 
\begin{equation}\label{ineq:difféo}
    \|\partial_I\Phi(n,I)\cdot h\|_{\infty} \geq \|h\|_\infty.
\end{equation}

Therefore, $\partial_I\Phi(n,I)$ is a global diffeomorphism on $\mathcal{C}^0([0,T],\R)$ and according to \eqref{ineq:difféo}
$$
\big\|\big(\partial_I\Phi(n,I)\big)^{-1}\big\|\leq 1.$$Hence, according to the implicit function Theorem, we can find a function $F\in \mathcal{C}^1\big(X, \mathcal{C}^0([0,T],\mathbb{R}^+)\big)$ such that 
$$\forall n \in X,  \qquad \Phi(n,F(n)) = 0.$$
i.e.
\begin{equation}\label{eq:fonction_implicite}
    \forall t\in [0,T], F(n)(t) = \int_0^\infty r(F(n)(t),x)n(t,x)dx.
\end{equation}
Finally
$$\|\partial_nF(n)\|\leq \big\|\big(\partial_I\Phi(n,I)\big)^{-1}\big\|\:\big\|\partial_n\Phi(n,I)\big\|\leq 1\times r_M.$$
Hence, $F$ is $r_M$-lipschitzian on $X$.

Now, if $\mu_1, \mu_2$ are two solutions of \eqref{eq:age}, associated to the activity function $I_1,I_2 \in \mathcal{C}^0([0,T], \R)$ with $\mu_1(0) = \mu_2(0) = n_\text{ini}$, then, according to Lemma \ref{lem: Estimation n par rapport a r}, 
$$\forall t \in [0,T],\qquad \|\mu_1(t)-\mu_2(t)\|_1 \leq \|\mu_1(0)-\mu_2(0)\|_1+2\|\partial_Ir\|_\infty\int_0^t\|I_1(\tau)-I_2(\tau)\|_\infty d\tau$$
Then, since $\mu_1(0)= \mu_2(0) = n_\text{ini}$, we obtain
$$\forall t \in [0,T],\qquad \|\mu_1(t)-\mu_2(t)\|_1\leq 2T\|\partial_Ir\|_\infty \|I_1-I_2\|_\infty.$$
And so
$$\|\mu_1-\mu_2\|_X\leq 2T\|\partial_Ir\|_\infty \|I_1-I_2\|_\infty.$$

Finally, if $n_1, n_2\in X$, and $\mu(n_1), \mu(n_2)\in X$ are the two solutions associated to the activity functions $I_1 = F(n_1)$ and $I_2 = F(n_2)$ and with $\mu_1(0) = \mu_2(0) = n_\text{ini}$, then,
$$\begin{aligned}\|\mu(n_1)-\mu(n_2)\|_X&\leq 2T\|\partial_Ir\|_\infty \|I_1-I_2\|_\infty. \\ 
&\leq 2T\|\partial_Ir\|_\infty  r_M\|n_1-n_2\|_X
\end{aligned}$$
Now, if we take $T>0$ such that $2T\|\partial_Ir\|_\infty  r_M<1$, then the function $n\mapsto \mu(n)$ is a contraction on $X$ and therefore, according to Picard Fixed point Theorem, there exist a unique function on $X$ such that $\mu(n)= n$. This function is the solution of \eqref{eq:E} on $[0,T]$. \\

Finally, we use the same argument on every interval $[kT, (k+1)T]$, with $k\in \mathbb{N}$ to prove that the solution $n$ is uniquely and well-defined on $\mathbb{R}^+$.
\end{proof}

\begin{lemma}\label{lem: Estimation n par rapport a r}
    Let $\mu_1$ and $\mu_2$ be the solution of \eqref{eq:age} with respectively $r_1(t,x):=r(x,I_1(t))$ and $r_2(t,x) = r(x,I_2(t))$, $\mu_1(0)=\mu_1^0$ and $\mu_2(0)=\mu_2^0$. 
    Then, 
    \[
    \fa t \geq 0, \|\mu_1(t)-\mu_2(t)\|_1\leq \|\mu_1^0-\mu_2^0\|_1+2\|\partial_Ir\|_\infty\int_0^t\|I_1(\tau)-I_2(\tau)\|_\infty d\tau.
    \]

\end{lemma}
\begin{proof}
The function $|\mu_1-\mu_2|$ is solution of the following equation
$$
\frac{\partial |\mu_1-\mu_2|}{\partial t} + \frac{\partial |\mu_1-\mu_2|}{\partial x}+ [r_1\mu_1-r_2\mu_2]\text{sg}(\mu_1-\mu_2) = 0.
$$
Integrating it on $[0,\infty)$, we obtain

$$\begin{aligned}\frac{d}{dt}\|\mu_1-\mu_2\|_1&=\int_0^\infty [r_1\mu_1-r_2\mu_2]\big(1-\text{sg}(\mu_1-\mu_2)\big)dx\\
&= \int_0^\infty [r_1(\mu_1-\mu_2)+(r_1-r_2)\mu_2]\big(1-\text{sg}(\mu_1-\mu_2)\big)dx\\
&=\int_0^\infty r_1\big((\mu_1-\mu_2)-|\mu_1-\mu_2|\big)+(r_1-r_2)\mu_2\big(1-\text{sg}(\mu_1-\mu_2)\big)dx\\
&\leq \int_0^\infty2|r_1-r_2|\mu_2dx\\
&\leq 2\|r_1-r_2\|_\infty \\
&\leq2\|\partial_Ir\|_{\infty} \|I_1(t)-I_2(t)\|_\infty.
\end{aligned}$$
By integrating on $[0,t]$, we finally get

\[
    \fa t \geq 0, \|\mu_1(t)-\mu_2(t)\|_1\leq \|\mu_1^0-\mu_2^0\|_1+2\|\partial_Ir\|_\infty\int_0^t\|I_1(\tau)-I_2(\tau)\|_\infty d\tau.
    \]
\end{proof}


\noindent \textbf{Acknowledgment} DS has received support from ANR ChaMaNe No: ANR-19-CE40-0024. 



\end{document}